\documentclass{amsart}
\usepackage{mathrsfs,amssymb,multirow,setspace}
\usepackage[a4paper, total={7in, 9in}]{geometry}
\theoremstyle{plain}
\usepackage{graphicx}
\newtheorem{theorem}{Theorem}[section]
\newtheorem{lemma}[theorem]{Lemma}
\newtheorem{proposition}[theorem]{Proposition}
\newtheorem{corollary}[theorem]{Corollary}

\theoremstyle{definition}

\newtheorem{example}[theorem]{Example}

\theoremstyle{remark}
\newtheorem{remark}[theorem]{Remark}

\input xy
\xyoption{all}

\begin{document}
\title[Equality of various graphs on  finite semigroups]{Equality of various graphs on  finite semigroups}
\author[Sandeep Dalal, Jitender Kumar]{Sandeep Dalal, Jitender Kumar}
\address{Department of Mathematics, Birla Institute of Technology and Science Pilani, Pilani, India}
\email{deepdalal10@gmail.com,jitenderarora09@gmail.com}

\begin{abstract}
In this paper, we consider various graphs, namely: power graph, cyclic graph, enhanced power graph and commuting graph, on a finite semigroup $S$. For an arbitrary pair of these four graphs, we classify finite semigroups such that the graphs in this pair are equal. In this connection, for each of the graph we also give a necessary and sufficient condition on $S$ such that it is complete.  The work of this paper generalize the corresponding results obtained for groups.
\end{abstract}

\subjclass[2010]{05C25}

\keywords{Monogenic semigroup, power graph, cyclic graph, enhanced power graph, commuting graph}

\maketitle

\section{Introduction}
The investigation of graphs associated to semigroups is a large research area. In 1964, Bos$\acute{\text{a}}$k \cite{b.Bosak} studied certain graphs over semigroups. Probably the most important class of graphs defined by semigroups is that of Cayley graphs (cf. \cite{a.Bevancayley,a.budden1985cayley,a.Cheninfintefamily2018,b.algebrasKelerav,a.trotter1978cartesian,a.witte1984survey}), since they have numerous applications (cf. \cite{a.kelarev2004cayleylabelled,a.kelarev2009cayley}). The concept of a (directed) power graph was first introduced by Kelarev and Quinn \cite{a.kelarev2000groups}. As explained in the survey \cite{a.abawajy2013power} and the book \cite{b.algebrasKelerav}, it is a standard practice to consider the class of undirected graphs as a subclass of the class of (directed) graphs (cf. \cite{b.algebrasKelerav}). Namely, a (directed) graph $\mathcal G = (V, E)$ is said to be \emph{undirected} if the implication $(u, v) \in E \Rightarrow (v, u) \in E$ holds, for all $u,v \in V$. Therefore, the  definition given in \cite{a.kelarev2000groups} also defined the power graph of all undirected graphs.  Further, the power graphs associated with groups and semigroups was studied in \cite{a.kelarev2002directed,a.kelarev2001powermatrices,a.kelarev2004semigroup}, and then in \cite{a.Cameron2010,a.Cameron2011,a.MKsen2009}. Recall that the undirected power graph of $S$ is the simple graph whose vertex set is $S$ and two distinct vertices $x, y$ are adjacent if either $x = y^m$ or $y = x^n$, for some $m, n \in \mathbb{N}$. The survey \cite{a.abawajy2013power} reports the current state of knowledge on the power graph associated with groups and semigroups.

Afkhami et al. \cite{a.afkhami2014cyclic} studied graph theoretic properties of the cyclic graph on a finite semigroup including its dominating number, independence number and genus. The \emph{cyclic graph} of a finite semigroup $S$ is the simple graph whose vertex set is $S$ and two distinct element $x, y$ are adjacent if $\langle x, y \rangle$ is a monogenic subsemigroup of $S$. The \emph{commuting graph} of a finite semigroup $S$ is the simple graph whose vertex set is $S$ and two distinct vertices $x, y$ are adjacent if $xy = yx$. The commuting graph of a finite group appears to be first studied by Brauer and Fowler in \cite{a.brauer1955groups} as a part of classification of finite simple groups. Since the elements of the center  are adjacent to all other vertices usually the vertices are assumed to be non-central. For more information on the commuting graphs of semigroups and groups, see \cite{a.Araujo2015,a.Araujo2011,a.Bauer2016} and the references therein.

In \cite{a.cameron2017}, Aalipour et al. characterized the finite groups such that the power graph of a finite group $G$ coincides with its commuting graph (with $G$ as a vertex set). If these two graphs of $G$ do not coincide, then to measure how much the power graph is close to the commuting graph of $G$, they introduced a new graph so called \emph{enhanced power graph} of a group $G$. The \emph{enhanced power graph} of a group $G$ is the simple graph whose vertex set is the group $G$ and two distinct vertices $x, y$ are adjacent if $x, y \in \langle z \rangle$ for some $z \in G$. In \cite{a.Bera2017}, Bera et al. studied the enhanced power graph of finite groups. Daniel et al. \cite{a.Dupont2017quotient} studied graph theoretic properties (connectivity, completeness etc.)  of the enhanced power graph of the quotient group $G/H$. Further, the rainbow connection number of the enhanced power graph of the group $G$ was calculated in \cite{a.Dupont2017}. The concept of the enhanced power graph on a semigroup can be defined analogously. To the best of our knowledge, the enhanced power graph on a finite semigroup is not studied so far.

In this paper, we initiate the study of enhanced power graph on a finite semigroup. The paper is arranged as follows. First, we provide necessary background material in Section 2. In Section 3, for each of the graph, viz. power graph, cyclic graph, enhanced power graph and commuting graph on a finite semigroup $S$, we provide a necessary and sufficient condition on $S$ such that it is complete. In Section 4, for an arbitrary pair of these four graphs, we classify finite semigroups such that the graphs in this pair are equal.

\section{Preliminaries}
\noindent We recall necessary definitions, results and notations of semigroup theory \cite{b.Howie} and graph theory \cite{b.West} which are used throughout in this paper. A \emph{semigroup} is a set with an
associative binary operation. A \emph{subsemigroup} of a semigroup is a subset that is also a
semigroup under the same operation. A semigroup $S$ is said to be \emph{commutative} if $xy = yx$ for all $x, y \in S$. An element $a$ of a semigroup $S$ is \emph{idempotent} if $a^2 = a$ and the set of all idempotents in $S$ is denoted by $E(S)$. For a subset $X$ of a semigroup $S$, the intersection of all the subsemigroups of $S$ containing $X$ is the smallest subsemigroup of $S$ containing $X$. It is denoted by $\langle X \rangle$ and known as subsemigroup generated by $X$. The subsemigroup $\langle X \rangle$ is the set of all the elements in $S$ that can be written as finite product of elements of $X$. If $X$ is finite then $\langle X \rangle$ is called finitely generated subsemigroup of $S$. A semigroup $S$ is called \emph{monogenic} if there exists $a \in S$ such that $S = \langle a \rangle$. Clearly, $\langle a \rangle = \{a^m \; : \; m \in \mathbb{N}\}$, where $\mathbb{N}$ is the set of positive integers.

For  $X \subseteq S$, the number of elements in $X$ is called the order of $X$ and it is denoted by $|X|$. The $\mathit{order}$ of an element $a\in S$, denoted by $o(a)$, is defined as  $|\langle a \rangle|$. In case of finite monogenic semigroup, there are repetitions among the powers of $a$. Then the set
\[\{x \in \mathbb{N} : (\exists \; y \in \mathbb{N}) a^x = a^y, x \ne y\}\]
is non-empty and so has a least element. Let us denote this least element by $m$ and call it the \emph{index} of the element $a$. Then the set
\[\{x \in \mathbb{N} \; : \; a^{m + x} = a^m \}\]
is non-empty and so it too has a least element $r$, which we call the \emph{period} of $a$. Let $a$ be an element  with index $m$ and period $r$. Thus, $a^m = a^{m + r}$. It follows that $a^m = a^{m + qr} \; \forall q \in \mathbb{N}$. By the minimality of $m$ and $r$ we may deduce that the powers
\[a, a^2, \ldots, a^m, a^{m + 1}, \ldots, a^{m + r-1}\]
are all distinct. For every $s \ge m$, by division algorithm we can write $s = m + qr + u$, where $q \ge 0$ and $0 \le u \le r-1$. then it follows that
\[a^s = a^{m + qr}a^u = a^m a^u = a^{m+u}.\]
Thus, $\langle a \rangle = \{a, a^2, \ldots, a^{m + r-1}\}$ and $o(a) = m + r - 1$. The subset \[K_a = \{a^m, a^{m+1}, \ldots, a^{m+r-1} \}\]  is a subsemigroup of $\langle a \rangle$. Indeed, $K_a$ is a cyclic subgroup of $\langle a \rangle$ with $|K_a| = r$ (cf. \cite{b.Howie}). Let $a$ be an element of a semigroup $S$ with index $m$ and period $r$. Then the monogenic semigroup $\langle a \rangle$ is denoted by $M(m, r)$. Also, sometimes $M(m, r)$ shall be written as
$\langle a : a^m = a^{m + r}\rangle$. The notations $m_a$ and $r_a$ denotes the index and period of $a$ in $S$, respectively. It is easy to observe that index of every element in a finite group $G$ is one. Consequently, for $a \in G$, we have $\langle a \rangle$ is the cyclic subgroup of $G$. The following results are useful in the sequel.

\begin{proposition}{\rm \cite[Proposition 1.2.3]{b.Howie}}
Every finite semigroup contains atleast one idempotent.
\end{proposition}

\begin{lemma}\label{gena-one idempotent}
Let $a$ be an element of a finite semigroup $S$. Then the subsemigroup $\langle a \rangle$ contains exactly one idempotent.
\end{lemma}

\begin{proof}
Let $m$ and $r$ be the index and period of $a$, respectively. Thus, 
$\langle a \rangle = \{a, a^2, \ldots, a^{m + r-1}\}$. The subgroup $K_a = \{a^m, a^{m+1}, \ldots, a^{m+r-1} \}$ of $\langle a \rangle$ contains exactly one idempotent. If for $1 \le i < m$, $a^i$ is an idempotent, then we have $a^{2i} = a^i$. Consequently, $m \le i$; a contradiction. Hence, the idempotent element of $K_a$ is the only idempotent in $\langle a \rangle$.
\end{proof}

\begin{lemma}\label{cyclic implies monogenic}
A cyclic subgroup of a finite semigroup $S$ is a monogenic subsemigroup of $S$.
\end{lemma}

\begin{proof}
Let $H$ be a cyclic subgroup of $S$. Then $H = \langle a \rangle$ for some $a \in S$. Since $H$ is finite so that $o(a) = n$ for some $n \in \mathbb N$. Thus $a^n = e$, where $e$ is the identity element of $H$. Consequently, $a^{-1} = a^{n-1}$. Now for any non-negative integer $k$, we get $a^{-k} = a^{k(n-1)}$. Thus, every element of $H$ is a positive power of $a$. Hence, $H$ is a monogenic subsemigroup of $S$.
\end{proof}


We also require the following graph theoretic notions. A graph $\mathcal{G}$ is a pair  $ \mathcal{G} = (V, E)$, where $V = V(\mathcal{G})$ and $E = E(\mathcal{G})$ are the set of vertices and edges of $\mathcal{G}$, respectively. We say that two different vertices $a, b$ are $\mathit{adjacent}$, denoted by $a \sim b$, if there is an edge between $a$ and $b$. It is clear that we are considering simple graphs, i.e. graphs with no loops or directed or repeated edges. A subgraph  of a graph $\mathcal{G}$ is a graph $\mathcal{G}'$ such that $V(\mathcal{G}') \subseteq V(\mathcal{G})$ and $E(\mathcal{G}') \subseteq E(\mathcal{G})$. A  subgraph $\mathcal{G}'$  of  graph $\mathcal{G}$ is said to be a $\emph{spanning subgraph}$ of $\mathcal{G}$ if $V(\mathcal{G}) = V(\mathcal{G}')$ and we shall write it as $\mathcal{G}' \preceq \mathcal{G}$. A graph $\mathcal{G}$ is said to be \emph{complete} if any two distinct vertices are adjacent. Let $S$ be a semigroup. The \emph{power graph} of $S$, denoted by ${\rm Pow}(S)$, is the simple graph whose vertex set is $S$ and two distinct elements are adjacent if one is a power of the other. The \emph{cyclic  graph} of $S$,  denoted by $\Gamma(S)$, is the simple graph whose vertex set is $S$ and two distinct element $x, y$ are adjacent if and only if $\langle x, y \rangle = \langle z \rangle$ for some $z \in S$. The \emph{enhanced power graph} of a semigroup $S$, denoted by $P_e(S)$, is the simple graph with vertex set $S$ and two distinct vertices are adjacent in $P_e(S)$ if there exists $z \in S$ such that $x, y \in \langle z \rangle$.  The \emph{commuting  graph} of $S$, denoted by $P_c(S)$, is the simple graph whose vertex set is $S$ and two different element $a, b$ are adjacent if $ab = ba$. The following result will be used at some points in the sequel.

\begin{theorem}{\rm \cite[Theorem 2.12]{a.MKsen2009}} \label{power graph of G - M K Sen}
Let $G$ be a finite group. Then ${\rm Pow}(G)$ is complete if and only if $G$ is a cyclic group of order $1$ or $p^m$, for some prime $p$ and $m \in \mathbb{N}$.
\end{theorem}

Throughout this paper $S$ is a finite semigroup, $G$ is a finite group, $C_n$ is the cyclic group of order $n$ and $\mathbb{N}_0 = \mathbb{N} \cup \{0\}$.

\section{Completeness of Graphs}
Let $\mathcal{K} = \{{\rm Pow}(S), \Gamma(S), P_e(S) , P_c(S) \}$ and $\Delta(S) \in \mathcal{K}$. In this section, we present a necessary and sufficient condition on $S$ such that $ \Delta(S)$ is complete. We begin with a relation between the elements of $\mathcal{K}$ in the following lemma.

\begin{lemma}\label{Spanning}
For a semigroup $S$, we have ${\rm Pow}(S) \preceq \Gamma(S) \preceq P_e(S) \preceq P_c(S)$.	
\end{lemma}

\begin{proof}
By \cite[Theoren 3.13]{a.afkhami2014cyclic}, note that Pow$(S) \preceq \Gamma(S)$. Now, suppose $a \sim b$  in $\Gamma(S)$. Then, for some $c \in S$, we have $\langle a, b \rangle = \langle  c \rangle$ . Consequently, $a, b \in \langle c \rangle$ so that $a  \sim b$  in $P_e(S)$. Thus, $\Gamma(S) \preceq P_e(S)$. In order to prove that $P_e(S)$ is a spanning subgraph of $P_c(S)$, suppose $a\sim  b$  in $P_e(S)$. Then $a, b \in \langle d \rangle$ for some $d \in S$. Since $\langle d \rangle$ is a commutative subsemigroup of $S$, we have  $ab = ba$ so that  $a \sim b$ in $P_c(S)$. Hence, $P_e(S) \preceq P_c(S)$.
\end{proof}

\begin{theorem}\label{enhanced-complete}
The enhanced power graph $P_e(S)$  is complete if and only if  $S$ is a monogenic semigroup.
\end{theorem}

\begin{proof} Let $S$ be a   monogenic semigroup. Then there exists $a \in S$ such that $S = \langle a \rangle$. For any  $x, y \in P_e(S)$, we have $x, y \in   \langle a \rangle$. Thus, by definition, $P_e(S)$ is complete. Conversely, suppose that  $P_e(S)$ is complete. Now choose an element $x \in S$ such that $o(x)$ is maximum. In order to prove that $S$ is monogenic, we show that  $S = H$, where $H = \langle x \rangle $. If $S \neq H$, then there exists $ y \in S$ but $y \notin H$. Since $P_e(S)$ is complete,  $x, y \in \langle z \rangle$ for some $z \in S$.  Also note that $\langle z \rangle = \langle x \rangle$. Consequently, $y \in \langle x \rangle$; a contradiction. Thus, $S= \langle x \rangle$. Hence, $S$ is a monogenic semigroup.
\end{proof}

\begin{theorem}\label{gamma-complete}
The cyclic graph $\Gamma(S)$ is complete if and only if one of the following holds:
\end{theorem}
\begin{enumerate}
\item[(i)] $S = \langle a  : a^{1 + r} = a \rangle$.
\item[(ii)]$ S =  \langle a  : a^{2+r} = a^2 \rangle$.
\item[(iii)] $ S = \langle a  : a^{3+r} = a^3 \rangle$ where $r$ is odd.
\end{enumerate}

\begin{proof}
Suppose that $\Gamma(S)$ is complete. Since $\Gamma(S)$ is a spanning subgraph of  $P_e(S)$ (cf. Lemma \ref{Spanning})  so that  $P_e(S)$ is complete. Consequently, $S$ is monogenic (cf. Theorem \ref{enhanced-complete}).
Clearly, $S = M(m,r)$ for some $m, r \in \mathbb{N}$. On contrary, suppose that $S$ is not of the form given in (i), (ii)  and  (iii).
Then either $S = M(3, r)$ such that $r$ is even or $S = M(m, r)$, where $m\geq 4$.

If $S = M(3, r)$ such that $r$ is even, then clearly $3 + r - 1$ and $3 + r + 1$ are even. Since $a^{3+r} = a^3$ implies $a^{4+r} = a^4$ so that $\langle a^2 \rangle = \{a^2, a^4, \ldots, a^{2+r}\}$. Note that $a^3 \notin \langle a^2 \rangle$. If $a^2 \in \langle a^3 \rangle$, then $a^2 = a^{3k}$ for some $k\in \mathbb N$. Thus $m \leq 2$; a contradiction for $m = 3$. Consequently, $a^2 \notin \langle a^3 \rangle$.
Let if possible $\langle a^2, a^3 \rangle = \langle a^t \rangle $ for some $a^t \in S$. We now show that no such $t \in \mathbb{N}$ exists. If $t = 1$, then $\langle a^2, a^3 \rangle  = \langle a \rangle$ so that
$a = a^l$, where $l \ge 2$. Thus, $m = 1$; a contradiction. If $t \in \{2, 3 \}$, then either $a^2 \in \langle a^3 \rangle$ or $a^3 \in \langle a^2 \rangle$; again a contradiction. Thus, we have $\langle a^2,a^3 \rangle = \langle a^t \rangle$ such that $t > 3$. Since $a^2 \in \langle a^2,a^3 \rangle = \langle a^t \rangle$ so that $a^2 = (a^t)^k$ for some $k \in \mathbb N$. Consequently, $ m \leq 2$; a contradiction. Thus, $\langle a^2, a^3 \rangle$ is not a monogenic subsemigroup of $S$ implies $a^2 $ is not adjacent  to $a^3$ in $\Gamma(S)$ so that $\Gamma(S)$ is not complete which is a contradiction. \\
We may now suppose $S = M(m, r)$, where $m \ge 4$. In this case, first note that $a, a^2, a^3, a^4$  all are distinct elements of $S$. Now, we show that $\langle a^2, a^3 \rangle$ is not a monogenic subsemigroup of $S$ so that $a^2$ and $a^3$ are not adjacent in $\Gamma(S)$, which is a contradiction of the fact $\Gamma(S)$ is complete. If possible, let  $\langle a^2, a^3 \rangle = \langle a^i \rangle$ for some $a^i \in S$. If $i = 1$, then $a \in \langle a^2, a^3 \rangle$. Thus, $a = a^t$, where $t \ge 5$ so that $m = 1$;  a contradiction. For $i = 2$, note that $a^3 \in \langle a^2 \rangle$ gives $a^3 = a^{2k}$ for some $k \ge 3$. Thus $ m \leq 3$; a contradiction. If $i \ge 3$, then $a^2 \in \langle a^i \rangle$, which implies that $a^2 = a^{ik}$ for some $k \in \mathbb N$. Thus $ m \leq 2$; again a contradiction.

Conversely, suppose $S$ is one of the form given in (i), (ii)  and  (iii). Thus, we have the following cases.

\textit{Case 1}: $S = M(1, r)$ i.e. $S = \{a, a^2, \ldots, a^r\}$. Since $S$ is a cyclic group, for any two distinct $x, y \in S$, note that $\langle x, y \rangle$ is a cyclic subgroup of $S$.
Consequently, $\langle x, y \rangle$ is a monogenic subsemigroup of $S$ (cf. Lemma \ref{cyclic implies monogenic}) so that $\langle x, y \rangle  = \langle z \rangle $ for some $z \in S$. Thus,  $x \sim y$ in $\Gamma (S)$.
Hence, $\Gamma(S)$ is complete.

\textit{Case 2}:  $S = M(2, r)$ i.e. $S = \{a, a^2, \ldots, a^{r+1}\}$ with $a^{2+r} = a^2$. Clearly, $K_a = \{a^2, a^3, \ldots, a^{r+1}\}$. For $2 \leq i \leq r+1$,
we have $a^i \in \langle a \rangle$ so that $\langle a^i, a \rangle = \langle a \rangle$. Thus $a \sim a^i$ in $\Gamma(S)$. Since $K_a$ is a cyclic subgroup of $S$ and for any $a^i, a^j \in K_a$,
the subsemigroup $\langle a^i, a^j \rangle$ is monogenic in $S$ so that $a^i \sim a^j$ in $\Gamma(S)$. Thus, $\Gamma(S)$ is complete.

\textit{Case 3}: $S =M(3, r)$ such that $r$ is odd. Clearly,  $S = \{a, a^2, a^3, \ldots, a^{2+r} \}$ with $a^{3+r} = a^3 $.
By the similar argument used in \textit{Case 2}, note that for $ 2 \leq i \leq 2+r$, we have $a \sim a^i$  in $\Gamma(S)$. Since $3 + r$ is even implies $3 + r = 2k$ for some $k \in \mathbb{N}$.
Thus $a^3 = a^{3+r} = a^{2k} = (a^2)^k$ so that $a^3 \in \langle a^2 \rangle$. Consequently, $a^5, a^7, \cdots, a^{2+r} \in \langle a^2 \rangle$.
For $i > 2$, note that $\langle a^2, a^i \rangle = \langle a^2 \rangle $ and it gives $a^2 \sim a^i$ in $\Gamma(S)$. Now, $K_a = \{a^3, a^4, \ldots, a^{2+r} \}$ is a cyclic subgroup of $S$ and for any $a^i, a^j \in K_a$ note that $\langle a^i, a^j \rangle$ is a monogenic subsemigroup of $S$. Thus $a^i \sim a^j$ in $\Gamma(S)$. Hence, $\Gamma(S)$ is complete.
\end{proof}

\begin{corollary}\label{Gamma(G)-complete}
The cyclic graph $\Gamma(G)$ is complete if and only if $G$ is a finite cyclic group.
\end{corollary}

\begin{theorem}\label{Pow-complete}
For a semigroup $S$, the following are equivalent:
\begin{enumerate}
\item[(i)] The power graph {\rm Pow$(S)$} is complete.
\item[(ii)] For some prime $p$ and $n \in \mathbb{N}_0$, we have $S = M(m, p^n)$ with either $m \in \{1, 2\}$ or $m = 3$ such that $3 + p^n$ is even.
\item[(iii)] The cyclic subsemigroups of $S$ are linearly ordered with respect to the usual containment relation (i.e., for any two cyclic subsemigroups $S_1, S_2$ of $S$, $S_1 \subseteq S_2$ or $S_2 \subseteq S_1$).
\end{enumerate}
\end{theorem}

\begin{proof}
(i)$ \Longleftrightarrow$ (ii). Suppose ${\rm Pow}(S)$ is complete. Then by Lemma \ref{Spanning}, $\Gamma(S)$ is complete. Consequently, $S = M(m, r)$, where either $m \in \{1,2\}$ or $m = 3$ with $3 + r$ is even (cf. Theorem \ref{gamma-complete}). Clearly, $|K_a| = r$ and $K_a$ is a cyclic subgroup of $S$. Note that ${\rm Pow}(K_a)$ is complete because ${\rm Pow}(S)$ is complete. Thus, by Theorem \ref{power graph of G - M K Sen}, we have $r = p^n$ for some prime $p$ and $n \in \mathbb{N}_0$.

Conversely, for some prime $p$ and $n \in \mathbb{N}_0$, suppose that $S = M(m, p^n)$ with either $m \in \{1, 2\}$ or $m = 3$ such that $3 + p^n$ is even. If $m = 1$, then $S$ is a cyclic group of order $p^n$ for some prime $p$. By Theorem \ref{power graph of G - M K Sen}, {\rm Pow$(S)$} is complete. If $m = 2$, then $K_a = \{ a^2, a^3, \ldots, a^{r+1}\}$ with $a^{2 + r} = a^2$ is a cyclic subgroup of order $r = p^n$ for some prime $p$. Consequently, ${\rm Pow}(K_a)$ is a complete graph (cf. Theorem \ref{power graph of G - M K Sen}). Also, for $2 \le i \le r + 1$, we have $a \sim a^i$ in ${\rm Pow}(S)$. Thus, Pow$(S)$ is complete. If $m = 3$ such that $3+p^n$ is even, then $K_a = \{a^3, a^4, \ldots, a^{2+p^n}\}$ with $a^{3 + p^n} = a^3$. Clearly,  $|K_a| = p^n$. By Theorem \ref{power graph of G - M K Sen}, ${\rm Pow}(K_a)$ is complete. Also, for $i \ge 2$ note that $a \sim a^i$ in Pow$(S)$. Since $3+p^n$  is even so that $a^3 \in  \langle a^2 \rangle$. Consequently, we have $\langle a^2 \rangle = S \setminus \{a\}$. Thus, for all $i > 2$ we have $a^2 \sim a^i$ in ${\rm Pow}(S)$. Hence, Pow$(S)$ is complete.

 (i)$ \Longleftrightarrow$ (iii) holds by \cite[Proposition 2.11]{a.MKsen2009}
\end{proof}

The definition of $P_c(S)$ gives us the following straightforward lemma.

\begin{lemma}\label{Commutative-complete}
The commuting graph $P_c(S)$ of a semigroup $S$ is complete if and only if $S$ is commutative.
\end{lemma}

\section{Equality of graphs}
In this paper, we consider various graphs, viz. power graph, cyclic graph, enhanced power graph and commuting graph, on a finite semigroup. In view of Lemma \ref{Spanning}, it would be interesting to investigate the following question.

\vspace{.2cm}
\textbf{Question:} For which (finite) semigroups the graphs in an arbitrary pair of these graphs are equal ?
\vspace{.2cm}

In this section, we answer the above question. In case of finite groups, this question was investigated in \cite{a.cameron2017}. We begin with an example of a semigroup whose cyclic graph and enhanced power graphs are not equal.

\begin{example}
Let $S = M(3,2) = \{a, a^2, a^3, a^4\}$, where $a^5 = a^3$. Note that $P_e(S)$ is complete (cf. Theorem \ref{enhanced-complete}) but $\Gamma(S)$ is not complete (cf. Theorem \ref{gamma-complete}). Then $P_e(S) \ne \Gamma(S)$. See Figure $1$.
\begin{figure}[h!]
	\centering
	\includegraphics[width=0.8\textwidth]{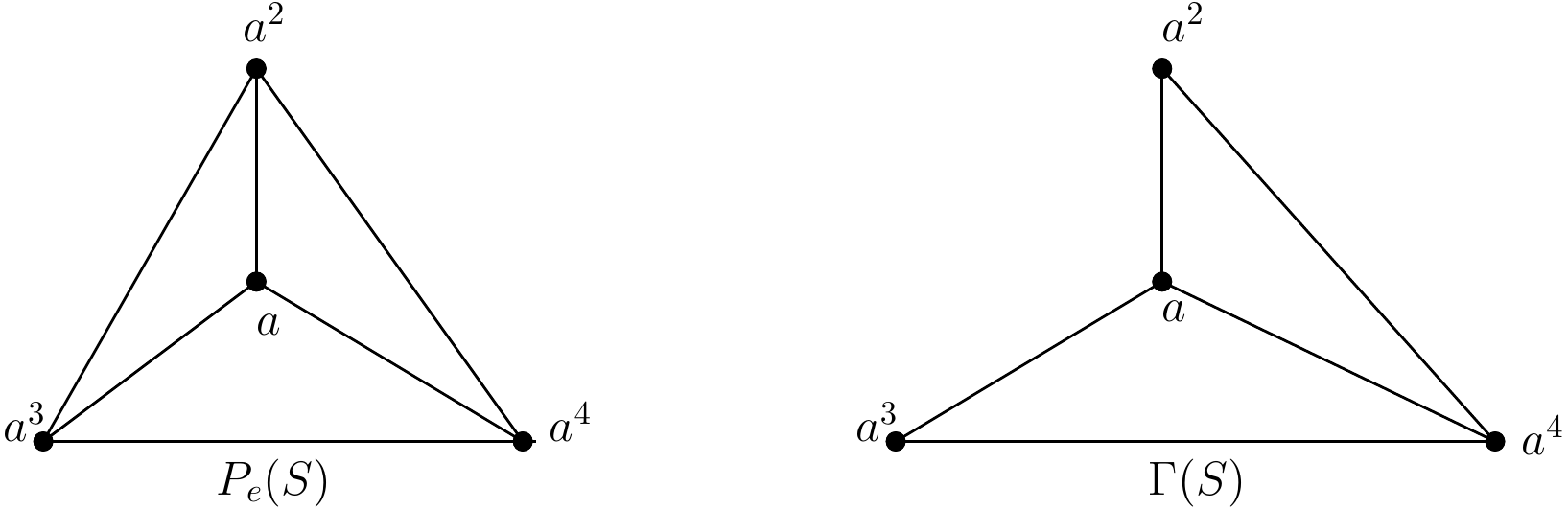}
	\caption{}
\end{figure}
\end{example}

\begin{theorem}\label{Gamma-equal-enhanced}
The enhanced power graph $P_e(S)$ is equal to $\Gamma(S)$ if and only if for each $a \in S$, we have one of the following form:
\begin{enumerate}
\item[(i)]$\langle a \rangle = \langle a  : a^{1 + r} = a \rangle$.
\item[(ii)]$\langle a \rangle = \langle a  : a^{2+r} = a^2 \rangle$.
\item[(iii)]$\langle a \rangle = \langle a  : a^{3+r} = a^3 \rangle$ where $r$ is odd.
\end{enumerate}
\end{theorem}

\begin{proof}
First suppose that $P_e(S) = \Gamma(S)$. On contrary, suppose that there exists $a \in S$ such that $\langle a \rangle $ is not of the form given in (i), (ii) and (iii). Then either $\langle a \rangle = M(3,r)$ with $r$ is even or $\langle a \rangle = M(m,r)$, where  $m \geq 4$. If $\langle a \rangle = M(3,r)$ with $r$ is even, then by the similar argument used in the proof of Theorem \ref{gamma-complete}, note that $a^2$ and $a^3$ are not adjacent in $\Gamma(S)$. Since $a^2, a^3 \in \langle a \rangle$, we have $a^2 \sim a^3$ in $P_e(S)$. Thus $P_e(S) \neq \Gamma(S)$; a contradiction. If $\langle a \rangle = M(m, r)$ with $m \ge 4$, then again by the proof of Theorem \ref{gamma-complete}, note that $a^2$ is not adjacent to $a^3$ in $\Gamma(S)$. Clearly,  $a^2 \sim a^3$ in $P_e(S)$. Consequently, $P_e(S) \neq \Gamma(S)$; again a contradiction. Hence, for each $a \in S$, $\langle a \rangle$ must be one of the form given in (i), (ii) and (iii).

Conversely, suppose that for each $a \in S$, $\langle a \rangle$ is one of the form given in (i), (ii) and (iii). Since $\Gamma(S)$ is a (spanning) subgraph of $P_e(S)$ (cf. Lemma \ref{Spanning}), it is sufficient to show that for any $x, y \in S$ such that $x \sim y$ in $P_e(S)$, we have $x \sim y$ in $\Gamma(S)$. Let $x \sim y$ in $P_e(S)$. Then there exists $z \in S$ such that $x, y \in \langle z \rangle$. By the hypothesis, $\langle z \rangle$ is one of the form given in (i), (ii) and (iii). By Theorem \ref{gamma-complete}, $\Gamma(\langle z \rangle)$ is complete. Consequently, $\langle x, y \rangle$ is a monogenic subsemigroup of $S$. Hence, $x \sim y$ in $\Gamma(S)$.
\end{proof}

\begin{corollary}\label{Gamma(G)-equal-P_e(G)}
For a finite group $G$, we have $\Gamma(G) = P_e(G)$.
\end{corollary}

\begin{example}
Let $S = M(2, 6) = \{a, a^2, a^3, a^4, a^5, a^6, a^7 \}$ where $a^8 = a^2$. By Theorem \ref{gamma-complete}, $\Gamma(S)$ is complete. Further, note that neither $a^2 \in \langle a^3 \rangle$
nor $a^3 \in \langle a^2 \rangle$. Thus, $a^2$ and $a^3$ is not adjacent in ${\rm Pow}(S)$. Hence, $\Gamma(S) \ne {\rm Pow}(S)$. See Figure $2$.

\begin{figure}[h!]
		\centering
		\includegraphics[width=0.8\textwidth]{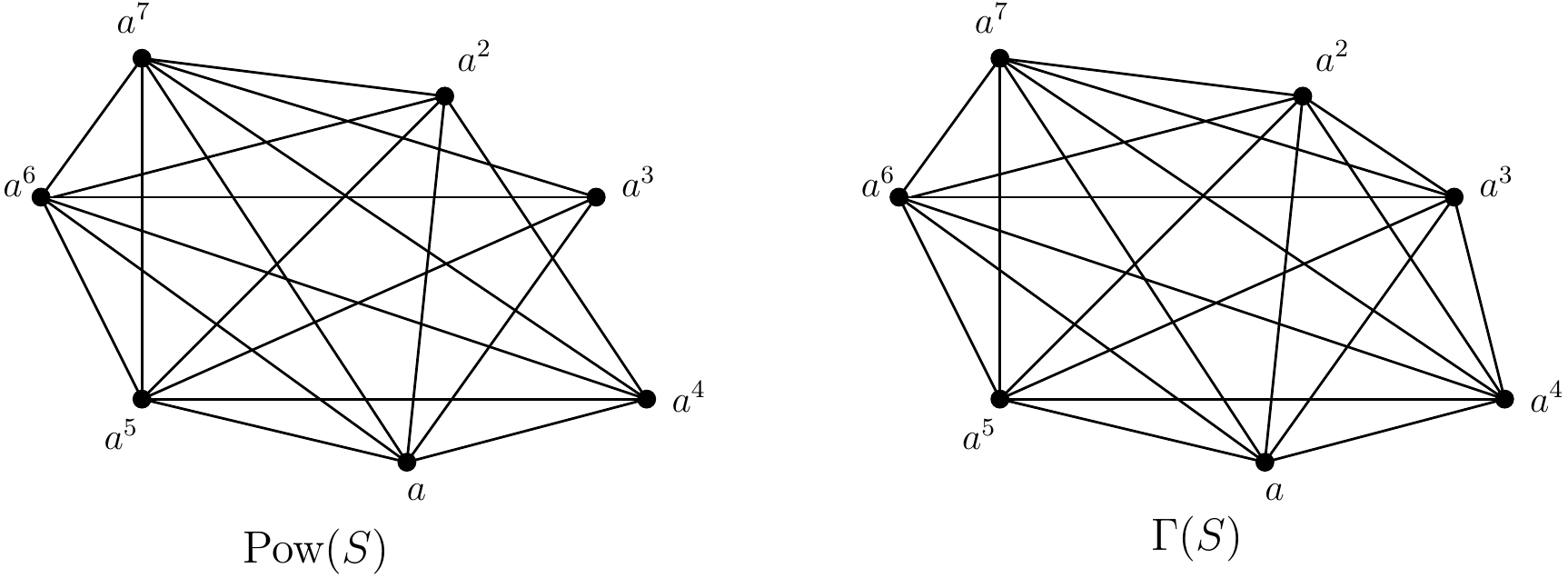}
		\caption{}
	\end{figure}
\end{example}

\begin{lemma}\label{gamma S-Pow S}
Let $S  = M(m, r)$ be a  monogenic semigroup. If $i < m$ and $a^i$ is not adjacent to $a^j$ in ${\rm Pow}(S)$, then $a^i$ is not adjacent to $a^j$ in $\Gamma (S)$.
\end{lemma}

\begin{proof}
If possible, let $a^i \sim a^j$ in $\Gamma (S)$. Then $\langle a^i, a^j \rangle = \langle a^k \rangle$ for some $k \in \mathbb N$. Note that  $k \neq i, j$. Otherwise, we have $a^i \sim a^j$ in Pow$(S)$; a contradiction.
Now we have the following cases on $i, j, k$:

\textit{Case 1}: $i,j < k$. Since  $a^i \in \langle a^k \rangle$, we have  $a^i = a^{tk}$ for some $t \in \mathbb N$. It follows that $m \leq i$; a contradiction .

\textit{Case 2}: $k < i,j$. Since $a^k \in \langle a^i, a^j \rangle$, we get  $m \leq k < i$; a contradiction.

\textit{Case 3}: $i < k < j$. Since $a^i \in \langle a^k \rangle$, we have  $a^i = a^{tk}$ for some $t \in \mathbb N$. Consequently, $m \leq i$; again a contradiction.

\textit{Case 4}: $j < k < i$. Since $a^j \in \langle a^k \rangle$, we get $a^j = a^{tk}$ for some $t \in \mathbb N$. Consequently,  $m \leq j < k <i$; again a contradiction.
\end{proof}

\begin{theorem}\label{power-cyclic-eq}
For a semigroup $S$, the following are equivalent:
\begin{enumerate}
\item[(i)] The cyclic graph  $\Gamma (S)$ is equal to  ${\rm Pow(S)}$.
\item[(ii)] Every cyclic subgroup of $S$ has a prime power order.
\item[(iii)] For each $a \in S$, we have $\langle a \rangle = M(m, p^n)$ for some prime $p$ and $m, n \in \mathbb N_0$.
\end{enumerate}
\end{theorem}

\begin{proof}
$({\rm i}) \Longleftrightarrow ({\rm iii})$. First assume that for each $a \in S$, we have $\langle a \rangle = M(m, p^n)$  for some prime $p$ and $m, n \in \mathbb N_0$. In view of  Lemma \ref{Spanning}, it is sufficient to show that $\Gamma(S)$ is a subgraph of $\text{Pow}(S)$. Let $a \sim b$  in $\Gamma(S)$. Then  $\langle a, b \rangle  = \langle c \rangle$ for some $c \in S$. Thus  $\langle c \rangle  = M(m, p^n)$ for some prime $p$ and $m, n \in \mathbb N_0$.
By Theorem \ref{power graph of G - M K Sen}, ${\rm Pow}(K_c)$ is complete. If $a, b \in K_c$, then $a \sim b$ in ${ \rm Pow}(K_c)$ so that one of $a, b$ is power of other. Thus, $a \sim b$ in ${ \rm Pow}(S)$. Without loss of generality, assume that $a \notin K_c$. Since $a, b \in \langle c \rangle$, we have $a = c^i$ and $b = c^j$ such that $i < m$. By Lemma \ref{gamma S-Pow S}, we have $a \sim b$ in ${\rm Pow}(S)$.

Conversely, suppose that ${\rm Pow}(S)  = \Gamma(S)$. For $a \in S$, clearly $\langle a \rangle = M(m, r)$ for some $m, r \in \mathbb{N}$. Then it is routine to verify ${\rm Pow}(K_a) = \Gamma(K_a)$. Since $\Gamma(K_a)$ is complete (cf. Corollary \ref{Gamma(G)-complete}) so is ${\rm Pow}(K_a)$. By Theorem \ref{power graph of G - M K Sen}, we have  $|K_a| = p^n$  for some prime $p$ and $n \in \mathbb N_0$. Thus $ r =  p^n$  for some prime $p$ and $n \in \mathbb N_0$.

$({\rm ii}) \Longleftrightarrow ({\rm iii})$. Suppose every cyclic subgroup of $S$ has prime power order. For $a \in S$, we have $\langle a \rangle = M(m, r)$. Since $K_a$ is a cyclic subgroup of $S$ of order $r$, we have $r = p^n$ for  some prime $p$ and $n \in \mathbb N_0$. Thus, $\langle a \rangle = M(m, p^n)$ for some prime $p$ and $n \in \mathbb N_0$. Conversely, let $H$ be a cyclic subgroup of $S$ so that $H = \langle a \rangle$ for some $a \in S$. Clearly, $H = M(1, r)$. By the hypothesis, we have $r = p^n$ for some prime $p$ and $n \in \mathbb N_0$. Thus, the order of $H$ is a prime power.
\end{proof}

In view of the Corollary \ref{Gamma(G)-equal-P_e(G)}, we have the following corollary of the above theorem.

\begin{corollary}{\rm \cite[Theorem 28 ]{a.cameron2017}}
For a finite group $G$, ${\rm Pow}(G)$ is equal to $P_e(G)$ if and only if every cyclic subgroup of $G$ has prime power order.
\end{corollary}

\begin{theorem}\label{enhanced-equal-power}
The enhanced power graph $P_e(S)$ is equal to  ${\rm Pow}(S)$ if and only if for each $a \in S$, we have either $\langle a \rangle = M(m, p^n)$ where $m \in \{1, 2\}$ or $\langle a \rangle  = M(3, p^n)$ such that $p$ is an odd prime.
\end{theorem}

\begin{proof}
Suppose that $\text{Pow}(S) = P_e(S)$. Since  ${\rm Pow}(S) \preceq \Gamma(S) \preceq P_e(S)$ (cf. Lemma \ref{Spanning}), we have $\Gamma(S) = P_e(S)$ and Pow$(S) = \Gamma(S)$. By Theorems \ref{Gamma-equal-enhanced} and \ref{power-cyclic-eq}, the result holds.
 \end{proof}

In general, the cyclic graph and the commuting graph of $S$ are not equal (see Example \ref{example cyclic not equal commuting}).  Now we present a necessary and sufficient condition on $S$ such that these two graphs are equal.

\begin{example}\label{example cyclic not equal commuting}
	Let $S = \mathbb Z_4 = \{\overline{0}, \overline{1}, \overline{2}, \overline{3}\}$ be a semigroup with respect to multiplication modulo $4$. Being a commutative semigroup, clearly $P_c(S)$ is complete but $\Gamma (S)$ is not complete as $\langle \; \overline{0}, \overline{1} \; \rangle$ is not a monogenic semigroup. Hence, $P_c(S) \ne \Gamma (S)$. See Figure $3$.
	\begin{figure}[h!]
		\centering
		\includegraphics[width=0.6\textwidth]{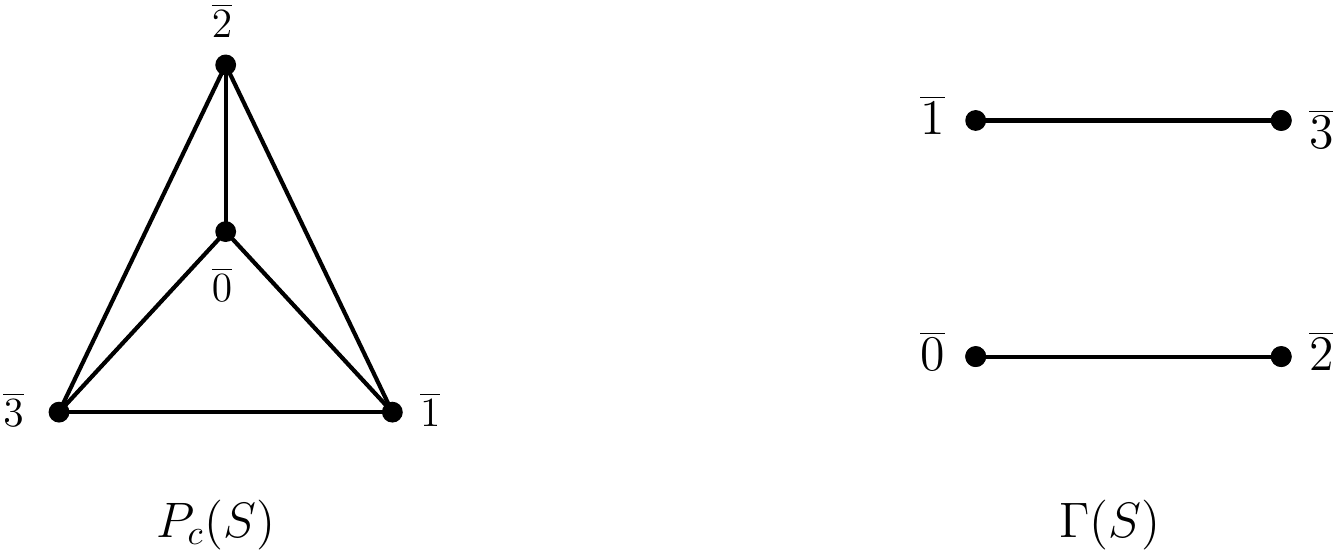}
		\caption{}
	\end{figure}
\end{example}

\begin{proposition}\label{cyclic graph-commuting graph-ff'=f'f}
If the cyclic graph $\Gamma(S)$ is equal to $P_c(S)$, then for $f, f' \in E(S)$ such that $ff' = f'f$, we have $f = f'$.
\end{proposition}

\begin{proof}
If possible, let $f \neq f'$. Since $ff' = f'f$, we have $f \sim f'$ in $P_c(S)$. Consequently, by the hypothesis, we have $f \sim f'$ in $\Gamma(S)$ which is a contradiction of the fact that every connected component of $\Gamma(S)$ contains exactly  one idempotent (cf. \cite[Theorem 2.3]{a.afkhami2014cyclic}). Hence, $f = f'$.
\end{proof}

\begin{proposition}\label{index}
If the cyclic graph $\Gamma(S)$ is equal to $P_c(S)$, then for each $a \in S$, we have either $\langle a \rangle = M(m, r)$ with $m \in \{1, 2\}$ or $\langle a \rangle = M(3, r)$ with $r$ is odd.
\end{proposition}

\begin{proof}
Suppose that $\Gamma(S) = P_c(S)$. If possible, for some $a \in S$, let $\langle a \rangle$ is not of the given form. Then either $\langle a \rangle = M(3, r)$ with $r$ is even or $\langle a \rangle = M(m, r)$ with $m \geq 4$. Then by the similar argument used in the proof of Theorem \ref{gamma-complete}, in each of the case, we have $a^2$ is not adjacent to $a^3$. Clearly, $a^2 \sim a^3$ in $P_c(S)$. Thus $\Gamma(S) \neq P_c(S)$; a contradiction. Hence, the result holds.
\end{proof}

\begin{theorem}\label{cyclic-equal-commuting}
The cyclic graph $\Gamma(S)$ is equal to $P_c(S)$ if and only if every commutative subsemigroup of $S$ is monogenic.
\end{theorem}

\begin{proof}
Let $\Gamma(S) = P_c(S)$ and $H$ be an arbitrary commutative subsemigroup of $S$. First, we prove that $\Gamma(H) = P_c(H)$. For that, let $x \sim y$ in $P_c(H)$, thus $xy = yx$. Consequently, we have $x \sim y$ in $P_c(S) = \Gamma(S)$. Thus, $\langle x, y \rangle = \langle z \rangle$ for some $z \in S$. Because of $z \in \langle x, y \rangle$, we get $z \in H$. Therefore, $x \sim y$ in $\Gamma(H)$ so that $P_c(H)$ is a subgraph of $\Gamma(H)$. As a result, $\Gamma(H) = P_c(H)$ (cf. Lemma \ref{Spanning}). Since $H$ is commutative, we have $P_c(H)$ is complete and so is $\Gamma(H)$. By Theorem \ref{gamma-complete}, $H$ is monogenic.

Conversely, suppose that every commutative subsemigroup of $S$ is monogenic. In order to prove $\Gamma(S) = P_c(S)$, it is sufficient to show $ P_c(S)  \preceq \Gamma(S)$ (cf. Lemma \ref{Spanning}). Let $a, b \in S$ such that $a \sim b$ in $P_c(S)$, we have $ab = ba$. Consequently, $\langle a, b \rangle$ is a commutative subsemigroup of $S$. By the hypothesis, $\langle a, b \rangle$ is a monogenic subsemigroup of $S$. Thus, $a \sim b$ in $\Gamma(S)$. Hence, we have the result.
\end{proof}

\begin{example}
Let $S = \{-1, 0, 1\}$ be a semigroup with respect to usual multiplication. Note that $0 \sim 1$ in $P_c(S)$ (cf. Lemma \ref{Commutative-complete}) but there is no edge between $0$ and $1$ in $P_e(S)$. Thus,  $P_c(S) \ne P_e(S)$. See Figure $4$.
\begin{figure}[h!]
	\centering
	\includegraphics[width=0.5 \textwidth]{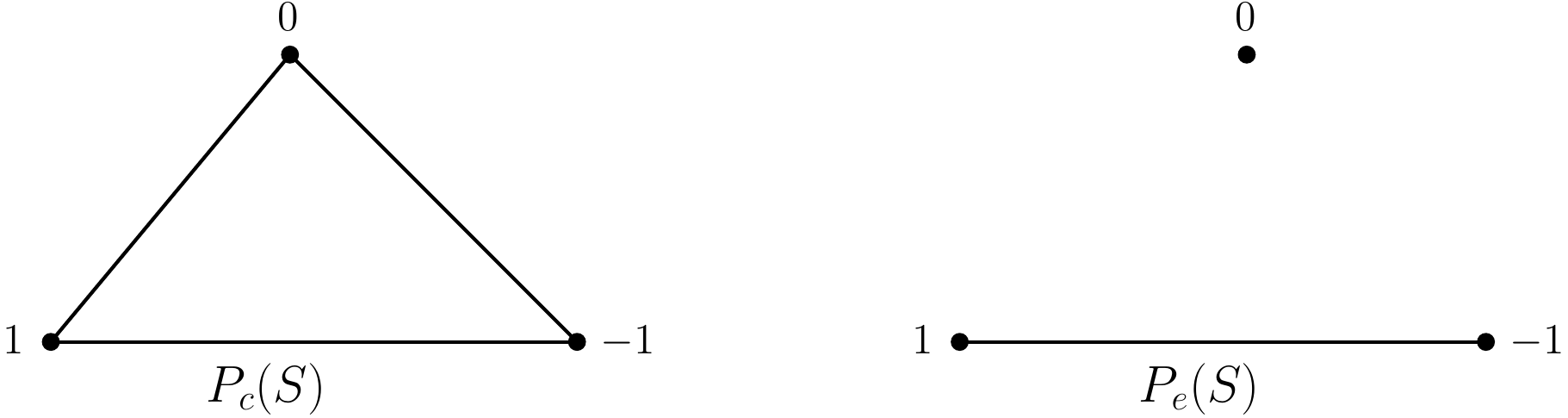}
	\caption{}
\end{figure}
\end{example}
	
\begin{remark}
Let $S$ be a commutative semigroup. Then $P_e(S) = P_c(S)$  if and only if $S$ is monogenic.
\end{remark}

Now for an idempotent $f$ in a semigroup $S$, define
$S_f = \{ a \in S \; : \; a^m= f \; \text{for some} \; m \in \mathbb N\}$.

\begin{remark}\label{classification-compoenent} Let $S$ be a finite semigroup. Then
$S = \underset{f \in E(S)}{\bigcup S_f}$ and for different $f, \;  f' \in E(S)$, we have $S_f \bigcap S_{f'} = \varnothing$.
\end{remark}

\begin{theorem}\label{enhanced = commuting}
The enhanced power graph $P_e(S)$ is equal to $P_c(S)$ if and only if the following  holds:
\begin{enumerate}\label{monogenic}
\item [(i)]  For $f, f' \in E(S)$ such that  $ff' = f'f$, we have  $f = f'$.
\item [(ii)] $S$ has no subgroup $C_p \times C_p$ for  prime $p$.
\item [(iii)] For $x, y \in S$ with  $xy = yx$ and atleast one of $m_x, m_y$ is greater than $1$, we have $x, y \in \langle z \rangle$ for some $z \in S$.
\end{enumerate}
\end{theorem}

\begin{proof}
First, suppose that $P_e(S) = P_c(S)$. In order to prove (i), let $f, f' \in E(S), \; ff' = f'f$ so that $f \sim f'$ in $P_c(S)$. Since $P_e(S) = P_c(S)$, we have $f \sim f'$ in $P_e(S)$. Thus $f, f' \in \langle z \rangle$ for some $z \in S$. Consequently, $f = f'$ (cf. Lemma \ref{gena-one idempotent}). Next, we shall show that  $S$ has no subgroup of the form $C_p \times C_p$ for some prime $p$. On contrary, we assume that $S$ has a subgroup $C_p \times C_p$ for some  prime $p$. For each $x \in C_p \times C_p$, we have  $o(x) = 1, \; p, \; p^2$. Since $C_p \times C_p$ is a non-cyclic subgroup of $S$, we get  $o(x) = p$ for all  $x \in (C_p \times C_p) \setminus \{\mathfrak{e}\}$, where $\mathfrak{e}$ is the identity element of the group $C_p \times C_p$.  For $x \in (C_p \times C_p) \setminus \{\mathfrak{e}\}$, we get $\langle x \rangle  \subsetneq C_p \times C_p$.  Thus, there exists $y \in C_p \times C_p$ such that $y \notin \langle x \rangle$. Note that $\langle x \rangle \cap \langle y \rangle = \{\mathfrak{e}\}$. Otherwise, if there exists a nonidentity element  $z \in \langle x \rangle \cap \langle y \rangle$, then we have $\langle z \rangle \subseteq \langle x \rangle$ and $\langle z \rangle \subseteq \langle y \rangle$. Since $o(x) = o(y) = o(z) = p$, we get $\langle x \rangle = \langle z \rangle  = \langle y \rangle$. Consequently,  $y \in \langle x \rangle$; a contradiction. Further, note that $|\langle x, y \rangle| =  |\langle x \rangle|\cdot |\langle y \rangle | = p^2 = |C_p \times C_p|$ and $\langle x, y \rangle \subseteq C_p \times C_p$, we get $\langle x, y \rangle = C_p \times C_p$. Thus, $xy = yx$ so that $x \sim y$ in $P_c(S)$. Since $P_e(S) = P_c(S)$, we get $x, y \in \langle z \rangle$ for some $z \in S$. Also, we have $x, y \in C_p \times C_p$ so that $m_x = m_y =1$. It follows that $x, y \in K_z$ which is a cyclic subgroup of $S$. Thus $\langle x, y \rangle = C_p \times C_p$ is a cyclic subgroup of $S$; a contradiction. Thus, (ii) holds. To prove (iii), let $x, y \in S, \; xy = yx$ and atleast one of $m_x, m_y$ is greater than $1$. Thus, $x \sim y$ in $P_c(S)$. Since $P_e(S) = P_c(S)$, we have $x \sim y$ in $P_e(S)$. Hence $x, y \in \langle z \rangle$ for some $z \in S$.

Conversely, suppose $S$ satisfies (i), (ii), and (iii). Since $P_e(S) \preceq P_c(S)$ (cf. Lemma \ref{Spanning}), we need to show that $P_c(S)$ is a subgraph of $P_e(S)$. Let $x \sim y$ in $P_c(S)$ so that $xy = yx$. If atleast one of $m_x, m_y$ is greater than $1$. Then by (iii), there exists $z \in S$ such that $x, y \in \langle z \rangle$  so that $x \sim y$ in $P_e(S)$. Hence, $P_e(S) = P_c(S)$. If $m_x = m_y = 1$, then $\langle x  \rangle$ and $\langle  y \rangle$ are the cyclic subgroup of $S$. Since, $x, y \in S$, by Remark \ref{classification-compoenent}, $x \in S_f, \; y \in S_{f'}$ for some $f, f' \in E(S)$. Then there exist $m, n \in \mathbb N$ such that $x^m = f, \; y^n = f'$. Note that $(xy)^{mn} = x^{mn} y^{mn} = (x^m)^n (y^n)^m = ff'$ and $(yx)^{mn} = f'f$. Since $xy = yx$, we have $ff' = f'f$. By (i), we get $f = f'$.
Since $f \in E(S)$ and $f \in \langle x \rangle \cap \langle y \rangle$, thus $f$ is the identity element of the subgroups $\langle x \rangle$ and $\langle  y \rangle$. Consequently,
 $f$ becomes the identity element of $\langle x, y \rangle$. Because of $xy = yx$, note that
 \[(x^iy^j)^{mn} = (x^i)^{mn}(y^j)^{mn} = (x^m)^{in}(y^n)^{jm} = f^{in} f^{jm}  = f\]
 so that $x^iy^j \in S_f$. As a result, we have $\langle x, y \rangle \subseteq S_f$. Thus $\langle x, y \rangle$ contains exactly one idempotent $f$. Since $\langle x, y \rangle$ is a finite monoid containing exactly one idempotent so that $\langle x, y \rangle$ becomes a subgroup of $S$, and hence is the direct product of two cyclic groups, say $C_r \times C_s$ for some $r,s \in \mathbb N$. Let gcd$(r, s) = d$. If $d = 1$, then $C_r \times C_s$  is a cyclic subgroup  which makes $\langle x, y \rangle$ to a cyclic subgroup of $S$. Consequently, $x, y \in \langle z \rangle$ for some $z \in S$. Thus, $x \sim y$ in $P_e(S)$.

 If $d > 1$, then there exists a prime $p$ such that $p$ divides $r$ and $s$. By Cauchy's theorem, there exist $x \in C_r$ and $y \in C_s$ such that $o(x) = o(y) = p$. Consequently, we get $(x, e_2)$ and $(e_1, y)$ in $C_r \times C_s$ such that $o(x, e_2) = o(e_1, y) = p$, where  $e_1, e_2$ are the identity elements of $C_r$ and $C_s$, respectively. Note that  $(x, e_2)$ and $(e_1, y)$ commute with each other and $\langle(x, e_2)\rangle  \cap \langle (e_1, y) \rangle  = \{(e_1, e_2)\}$. It follows that $|\langle (x, e_2), (e_1, y) \rangle | = p^2$. Now $\big( (x, e_2)^i (e_1, y)^j \big)^p = (e_1, e_2)$ so that there does not exist an element of order $p^2$ in the group $\langle (x, e_2), (e_1, y) \rangle$. Thus, $\langle (x, e_2), (e_1, y) \rangle$ is non-cyclic group of order $p^2$. Consequently, $\langle (x, e_2), (e_1, y) \rangle$ is of the form $C_p \times C_p$; a contradiction of (ii).
\end{proof}

Now, we have the following corollary of the above theorem.

\begin{corollary}{\rm \cite[Theorem 30]{a.cameron2017}}
Let $G$ be a finite group. Then the enhanced power graph $P_e(G)$ is equal to $P_c(G)$ if and only if $G$ has no subgroup $C_p \times C_p$ for prime $p$.
\end{corollary}

\begin{example}For any integer $n \ge 1$, let $[n] = \{1, 2, \ldots, n\}$. The \emph{Brandt semigroup} $(B_n, \cdot)$, where $B_n = ([n]\times[n])\cup \{ 0\}$ and the operation \lq$\cdot$\rq $\;$ is given by
	\[ (i,j) \cdot (k,l) =
	\left\{\begin{array}{cl}
	(i,l) & \text {if $j = k$;}  \\
	0    & \text {if $j \neq k $}
	\end{array}\right.  \]
	and, for all $\alpha \in B_n$, $\alpha \cdot 0 = 0 \cdot \alpha = 0$.
For $S = B_2$, note that $0 \sim (1, 1)$ in $P_c(S)$ whereas there is no edge between $0$ and $(1, 1)$ in ${\rm Pow}(S)$. Hence, $P_c(S) \ne {\rm Pow}(S)$. See Figure $5$.
	\begin{figure}[h!]
		\centering
		\includegraphics[width=0.7 \textwidth]{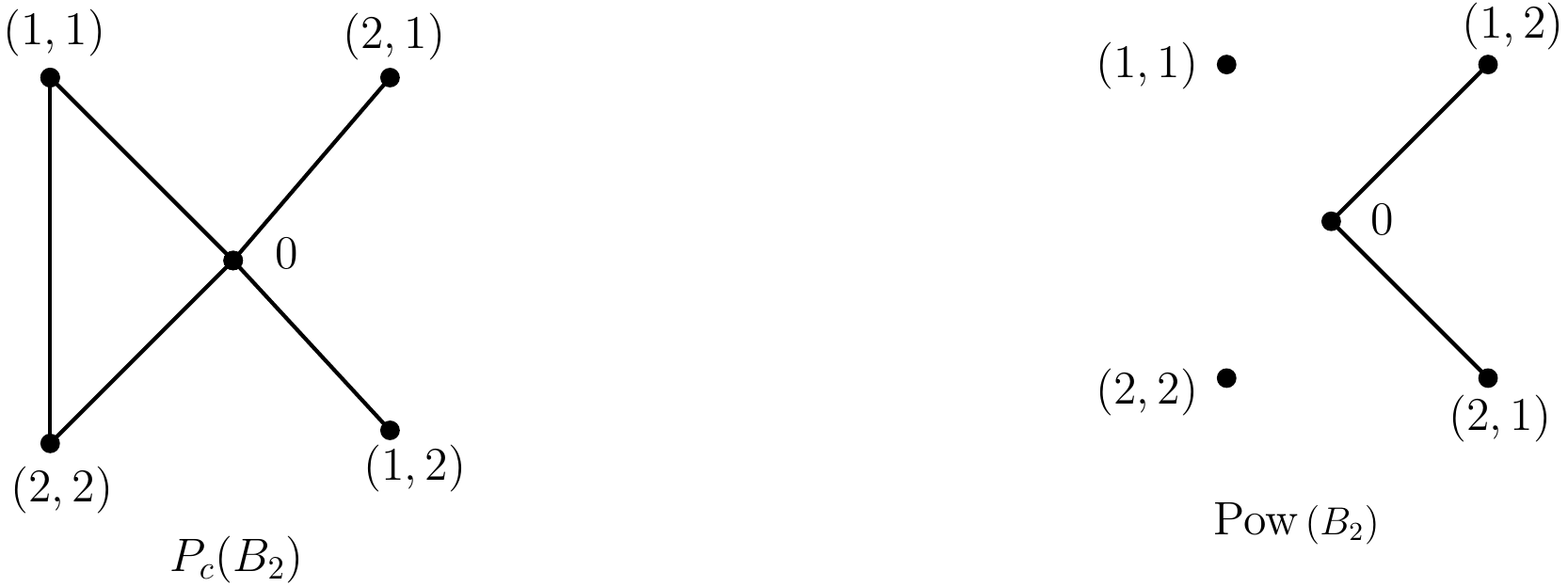}
		\caption{}
	\end{figure}
\end{example}

\begin{theorem}\label{Power = commuting}
The power graph ${\rm Pow}(S)$  is equal to the commuting graph $P_c(S)$ if and only if
\end{theorem}
\begin{enumerate}
\item[(i)] the order  of every cyclic subgroup of $S$ is of prime power.
\item[(ii)] every commutative subsemigroup of $S$ is monogenic.
\end{enumerate}

\begin{proof}
In view of Lemma \ref{Spanning}, the proof is strightforward by Theorems \ref{power-cyclic-eq} and \ref{cyclic-equal-commuting}.
\end{proof}

\begin{lemma}\label{monogenic implies no subgroup Cp}
Let $G$ be a nontrivial group and its every cyclic subgroup has prime power order. Then every commutative subgroup of $G$ is cyclic if and only if $G$ has no subgroup of the form $C_p \times C_p$, for some prime $p$.
\end{lemma}

\begin{proof}
Suppose every commutative subgroup of $G$ is cyclic. On contrary, let $G$ has a subgroup of the form $C_p \times C_p$, where $p$ is a prime. Then by the similar argument used in the proof of Theorem \ref{enhanced = commuting}, there exist $x, y \in C_p \times C_p$  such that $\langle x, y \rangle = C_p \times C_p$.  As a result, we have a commutative subgroup $\langle x, y \rangle$ of $G$ which is non-cyclic; a contradiction. Thus, $G$ has no subgroup of the form $C_p \times C_p$, for some prime $p$.

Conversely, suppose $G$ has no subgroup of the form $C_p \times C_p$. Let $H$ be an arbitrary commutative subgroup of $G$. To prove $H$ is cyclic i.e. $H = \langle x \rangle$, for some $x \in H$, we choose an element $x \in H$ such that $o(x)$ is maximum. First, we shall show that for an arbitrary $y \in H$, we have either $x \in \langle y \rangle$ or $y \in \langle x \rangle$. For $y \in H$, we get $\langle x, y \rangle$ is a commutative subgroup of $H$. Consequently, $ \langle x, y \rangle$ is a cyclic subgroup of $H$ (see proof of Theorem \ref{enhanced = commuting}). By the hypothesis, $|\langle x, y \rangle| = q^n$, where $q$ is a prime and $n \in \mathbb{N}$. Then by Theorem \ref{power graph of G - M K Sen}, ${\rm Pow}(\langle x, y \rangle)$ is complete  so that either $x \in \langle y \rangle$ or $y \in \langle x \rangle$. Now we claim that $H = \langle x \rangle$. If possible, let $\langle x \rangle \subsetneq H$. Then there exists $y \in H$ such that $y \notin \langle x \rangle$. We must have $x \in \langle y \rangle$. Because of $o(x)$ is maximum, we have $\langle x \rangle = \langle y \rangle$; a contradiction of $y \notin \langle x \rangle$. Hence, the subgroup $H$ is cyclic.
\end{proof}

In view of Lemma \ref{monogenic implies no subgroup Cp}, we have the following corollary of the Theorem \ref{Power = commuting}.

\begin{corollary}
The power graph {\rm Pow$(G)$} of a group $G$ is equal to $P_c(G)$ if and only if
\begin{enumerate}
\item[(i)] every cyclic subgroup of $G$ has prime power order.
\item[(ii)] $G$ has no subgroup of the form $C_p \times C_p$ for some prime $p$.
\end{enumerate}
\end{corollary}

\noindent\textbf{Acknowledgement:} The second authors wishes to acknowledge the support of MATRICS grant (MTR/2018/000779) funded by SERB.

\end{document}